\documentclass[11pt]{article}
\usepackage[ansinew]{inputenc}
\usepackage{graphicx}
\usepackage{color}
\usepackage{pdfsync}
\usepackage{enumerate,latexsym}
\usepackage{latexsym}
\usepackage{amsmath,amssymb}
\usepackage{graphicx}
\usepackage{amsthm}

\def\cL{\mathcal{L}}
\newfont{\bb}{msbm10 at 11pt}
\newfont{\bbsmall}{msbm8 at 8pt}

\def\S{{\Sigma}}

\newcommand{\ov}{\overline}
\def\rth{\mathbb{R}^3}
\def\R{\mathbb{R}}
\def\B{\mathbb{B}}
\def\N{\mathbb{N}}

\def\Q{\mathbb{Q}}

\newcommand{\ben}{\begin{enumerate}}
\newcommand{\bit}{\begin{itemize}}
\newcommand{\een}{\end{enumerate}}
\newcommand{\eit}{\end{itemize}}
\newcommand{\wh}{\widehat}
\newcommand{\Int}{\mbox{Int}}
\newcommand{\cH}{{\mathcal H}}

\newcommand{\wt}{\widetilde}

\newcommand{\ed}{\end{document}}

\def\a{{\alpha}}

\def\t{{\theta}}

\def\g{{\gamma}}
\def\G{{\Gamma}}

\def\de{{\delta}}

\def\ve{{\varepsilon}}

\newcommand{\cD}{{\mathcal D}}
\newcommand{\cB}{{\mathcal B}}

\newtheorem{theorem}{Theorem}[section]

\newtheorem{proposition}[theorem]{Proposition}

\newtheorem{corollary}[theorem]{Corollary}
\newtheorem{definition}[theorem]{Definition}

\newtheorem{assertion}[theorem]{Assertion}

\newtheorem{claim}[theorem]{Claim}

\definecolor{b}{rgb}{.1,.1,.7}
\definecolor{rr}{rgb}{.8,0,.3}
\definecolor{g}{rgb}{0,.5,0}
\definecolor{pp}{rgb}{.5,0,.7}
\definecolor{r}{rgb}{.6,0,.3}
\definecolor{y}{rgb}{.9,.99,.9}

\begin{document}

\begin{title}
{One-sided curvature estimates for H-disks }\end{title}
\date{}

\begin{author}
{William H. Meeks III\thanks{This material is based upon
   work for the NSF under Award No. DMS-1309236.
   Any opinions, findings, and conclusions or recommendations
   expressed in this publication are those of the authors and do not
   necessarily reflect the views of the NSF.}
   \and Giuseppe Tinaglia\thanks{The second author was partially
   supported by
EPSRC grant no. EP/M024512/1}}
\end{author}
\maketitle
\vspace{-.3cm}
\begin{abstract}
In this paper we prove an extrinsic one-sided curvature estimate for
disks embedded in $\rth$ with constant mean curvature, which is independent
of the value of the constant mean curvature.
We apply this extrinsic one-sided curvature estimate in~\cite{mt8}
to prove a weak chord arc  result for these disks. In Section~\ref{sec:consequences}
we apply this weak chord arc result  to obtain   an intrinsic
version of the  one-sided curvature estimate for
disks embedded in $\rth$ with constant mean curvature.
In a natural sense, these one-sided curvature estimates generalize respectively,
the extrinsic and intrinsic one-sided
curvature estimates for minimal disks embedded in $\rth$
given by Colding and Minicozzi in Theorem~0.2 of~\cite{cm23} and
in Corollary~0.8 of~\cite{cm35}.

\vspace{.3cm}


\end{abstract}

\section{Introduction.}
In this paper we prove a one-sided curvature estimate for disks embedded in $\rth$
with constant mean curvature.
An important feature of this estimate is its independence on the value of
the constant mean curvature.

For clarity of  exposition, we will call an oriented surface
$\Sigma$ immersed in $\rth$ an {\it $H$-surface} if it
is {\it embedded}, {\em connected}  and it
has {\it non-negative constant mean curvature $H$}. We will  call an
$H$-surface an {\em $H$-disk} if the $H$-surface is homeomorphic
to a closed unit disk in the Euclidean plane.  We remark that this definition of $H$-surface
differs from the one given in~\cite{mt7},
where we restricted to the case where $H>0$.
In this paper  $\B(R)$ denotes the open ball in $\rth$ centered at the origin $\vec{0}$
of radius $R$, $\ov\B(R)$ denotes its closure and for a
point $p$ on a surface $\Sigma$ in $\rth$, $|A_{\Sigma}|(p)$ denotes the norm
of the second fundamental
form of $\Sigma$ at $p$.

The main result of this paper, which is Theorem~\ref{th} below,  states
that if $\cD $ is an
$H$-disk which lies on one side of a plane $\Pi$, then the norm of the
second fundamental form of $\cD $ cannot
be arbitrarily large at points sufficiently far from the boundary of $\cD$
and sufficiently close to $\Pi$.  This
 estimate is inspired by, depends upon and provides a natural generalization of
 the Colding-Minicozzi
one-sided curvature estimate for minimal disks
embedded in $\rth$,  which is given in Theorem~0.2 in~\cite{cm23}.

\begin{theorem}[One-sided curvature estimate for $H$-disks] \label{th}
There exist $\ve\in(0,\frac{1}{2})$ and
$C \geq 2 \sqrt{2}$ such that for any $R>0$, the following holds.
Let $\cD$ be an $H$-disk such that $$\cD\cap \B(R)\cap\{x_3=0\}
=\O \quad \mbox{and} \quad \partial \cD\cap \B(R)\cap\{x_3>0\}=\O.$$
Then:
\begin{equation} \label{eq1}
\sup _{x\in \cD\cap \B(\ve R)\cap\{x_3>0\}} |A_{\cD}|(x)\leq \frac{C}{R}.
\end{equation} In particular, if $\cD\cap \B(\ve R)\cap\{x_3>0\}\neq\O$, then $H\leq \frac{C}{R}$.
\end{theorem}

In contrast to the minimal case, the constant $C$ in equation (\ref{eq1})
need not improve with smaller
choices of $\ve$. To see this, let
$S$ be the sphere of radius $\frac12$ centered at $(0,0,\frac12 )$.
Each surface in the sequence $E_n=(S+(0,0,\frac{1}{n}))\cap \ov{\B}(1)$
is a compact disk that
satisfies the hypotheses of the theorem for $R=1$, has $|A_{E_n}|=2\sqrt{2}$
and, as $n$ tends to infinity, $E_n$
moves arbitrarily close to the origin. In particular these examples show that
the constant $C$ in the above theorem must be at
least $2\sqrt{2}$ no matter how small $\ve$ is.

Theorem~\ref{th} plays an important role in deriving a weak  chord arc property
for $H$-disks in our papers~\cite{mt8, mt13},
which we describe in Section~\ref{sec:consequences}. This weak chord arc property
was inspired by and gives a generalization of
Proposition~1.1 in~\cite{cm35} by Colding and Minicozzi
 for 0-disks to the case of $H$-disks; we  apply this weak chord arc property  to obtain
an intrinsic version of the one-sided curvature estimate in Theorem~\ref{th},
which we describe in Theorem~\ref{TH}.  In the case $H=0$, this intrinsic one-sided
curvature estimate follows from Corollary~0.8 in~\cite{cm35} by Colding and Minicozzi.



\section{Preliminaries.} \label{sec:pre}

Throughout this paper, we use the following notation. Given $a,b,R>0$,
$p\in \rth$ and $\S$ a surface in $\rth$:

\bit
\item $\B(p,R)$ is the open ball of radius $R$ centered at $p$.
\item $\B(R)=\B(\vec{0},R)$, where $\vec{0}=(0,0,0)$.
\item For $p\in \S$, $B_{\S}(p,R)$ denotes the open intrinsic ball in $\S$ of radius $R$.
\item $C(a,b)=\{(x_1,x_2,x_3) \mid x_1^2+x_2^2\leq a^2,\, |x_3|\leq b\}$.
\item $A(r_1,r_2)=\{(x_1,x_2,0)\mid r_2^2\leq x_1^2+x_2^2\leq r_1^2\}$.
\eit

We next recall several results from our manuscript~\cite{mt7} that will be used in this paper.

We first  introduce the notion of multi-valued graph, see~\cite{cm22} for
further discussion and Figure~\ref{3-valuedgraph}.
\begin{figure}[h]
\begin{center}
\includegraphics[width=12cm]{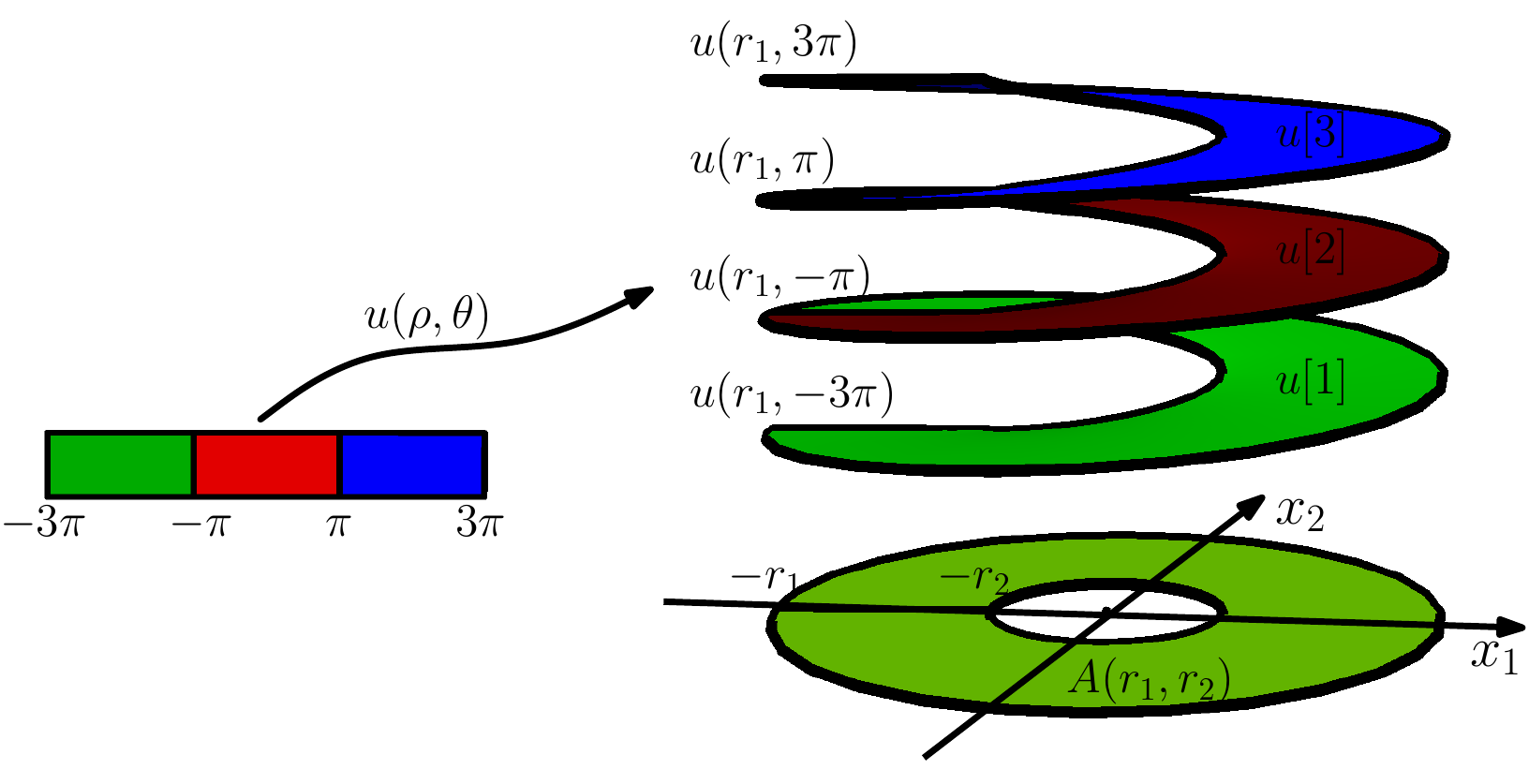}
\caption{A right-handed 3-valued graph.}
\label{3-valuedgraph}
\end{center}
\end{figure}
Intuitively, an $N$-valued graph is a simply-connected embedded surface covering an
annulus such that over a neighborhood of each point of the annulus, the
surface consists of $N$ graphs. The stereotypical  infinite multi-valued
graph is half of the helicoid, i.e., half of an infinite  double-spiral staircase.

\begin{definition}[Multi-valued graph]\label{multigraph} {\rm
Let $\mathcal{P}$ denote the universal cover of the
punctured $(x_1,x_2)$-plane, $\{(x_1,x_2,0)\mid (x_1,x_2)\neq (0,0)\}$, with global coordinates
$(\rho , \theta)$.
\ben[1.] \item
An {\em $N$-valued graph over the annulus $ A(r_1,r_2)$} is a single valued graph
$u(\rho, \theta)$ over $\{(\rho ,\theta )\mid r_2\leq \rho
\leq r_1,\;|\theta |\leq N\pi \}\subset \mathcal{P}$, if $N$ is odd, or
over $\{(\rho ,\theta )\mid r_2\leq \rho
\leq r_1,\;(-N+1)\pi\leq \theta \leq \pi (N+1)\}\subset \mathcal{P}$, if $N$ is even.
\item An  $N$-valued graph $u(\rho,\t)$ over the annulus $ A(r_1,r_2)$ is
called {\em righthanded} \, [{\em lefthanded}] if whenever it makes sense,
$u(\rho,\t)<u(\rho,\t +2\pi)$ \, [$u(\rho,\t)>u(\rho,\t +2\pi)$]
\item The set $\{(r_2,\theta, u(r_2,\theta)), \theta\in[-N\pi,N\pi]\}$
when $N$ is odd (or the set $\{(r_2,\theta, u(r_2,\theta)), \theta\in[(-N+1)\pi,(N+1)\pi]\}$
when $N$ is even) is the {\em inner boundary} of the $N$-valued graph.
\een }
\end{definition}

From Theorem 2.23 in~\cite{mt7} one obtains the following, detailed geometric description
of an $H$-disk with large
norm of the second fundamental form at the origin. The precise meaning
of certain statements below are
made clear in~\cite{mt7} and we refer the reader to that paper for further details.

\begin{theorem}\label{mainextension}
Given $\ve,\tau>0$   and $\ov{\ve}\in (0,\ve/4)$ there
exist  constants $\Omega_\tau:=\Omega(\tau )$,
$\omega_\tau:=\omega(\tau )$ and  $G_\tau:=G(\ve,\tau,\ov\ve ) $
such that if $M$ is an $H$-disk, $H\in (0,\frac 1{2\ve})$,
$\partial M\subset \partial \B(\ve)$, $\vec 0\in M$ and
$|A_M|(\vec 0)>\frac{1}{\eta}G_\tau$,  for $\eta\in(0,1]$, then for any $p\in \ov{\B}(\vec{0},\eta\ov{\ve})$ that is a maximum of the
function
$|A_{M}|(\cdot)(\eta\bar\ve-|\cdot|)$, after translating
$M$ by $-p$, the following geometric description of $M$ holds:
\par

On the scale of the norm of the second
fundamental form  $M$ looks like one or two helicoids nearby  the
origin and, after a rotation that turns these helicoids into
vertical helicoids, $M$  contains a 3-valued graph
$u$ over  $A(\ve\slash\Omega_\tau,\frac{\omega_\tau}{|A_M|(\vec 0)})$
with the norm of its gradient less than $\tau$ and with its inner boundary
in $\B(10\frac{\omega_\tau}{|A_M|(\vec 0)})$.
 \end{theorem}

  Theorem~\ref{mainextension}
 was inspired by the pioneering work of Colding and Minicozzi  in the minimal
 case~\cite{cm21,cm22,cm24,cm23};
however in the constant positive mean curvature setting this description has led to a
different conclusion, that is the existence of extrinsic radius and curvature estimates stated
below, which do not depend on the results in this paper.

\begin{theorem}[Extrinsic radius estimates, Theorem~3.4 in~\cite{mt7}] \label{rest} There exists an
${\mathcal R_0}\geq \pi$ such that for any $H$-disk $\cD$ with $H>0$,
$$ {\large{\LARGE \sup}_{\large p\in \cD }\{d_{\rth}(p,\partial \cD)\} \leq\frac{{\mathcal R_0}}{H}.}$$
\end{theorem}

\vspace{.1cm}

\begin{theorem}[Extrinsic curvature estimates, Theorem~3.5 in~\cite{mt7}] \label{ext:cest}
Given $\delta,\cH>0$, there exists a constant $K_0(\delta,\cH)$ such that
for any $H$-disk $\cD$ with $H\geq \cH$,
{\large $${{\sup}_{ \{p\in \cD \, \mid \, d_{\rth}(p,\partial
\cD)\geq \delta\}} |A_\cD|\leq  K_0(\delta,\cH)}.$$}
\end{theorem}

\vspace{.1cm}

Indeed, since the plane and the helicoid are complete simply-connected minimal
surfaces properly embedded in $\rth$, a radius estimate does not hold
in the minimal case. Moreover rescalings of a helicoid
give a sequence of embedded minimal disks with arbitrarily large norm
of the second fundamental form at points arbitrarily far
from their boundary curves; therefore in the minimal setting,
the extrinsic curvature estimates do not hold.

Next, we recall the notion of the
flux of a 1-cycle in an $H$-surface; see for instance the references~\cite{kks1,ku2,smyt1} for
further discussions of this invariant.

\begin{definition} \label{def:flux} {\rm
Let $\gamma$ be a 1-cycle in an $H$-surface $M$. The {\em flux} of
$\gamma$ is $\int_{\gamma}(H\gamma+\xi)\times \dot{\gamma}$, where $\xi$
is the unit normal to $M$ along $\gamma$. }
\end{definition}

The flux of a 1-cycle in   an $H$-surface $M$ is a homological invariant and
we say that  $M$ has {\em  zero flux} if the flux of any 1-cycle in $M$ is zero;
in particular, since the first homology group of a disk is zero,   an $H$-disk has  zero flux.

Finally, we recall the following definition.

\begin{definition} \label{def:lbsf} {\rm Let $U$ be an open set in $\rth$.
We say that a sequence of surfaces
$\Sigma(n)$ in $\rth$,  has {\em locally bounded norm of the second
fundamental form in $U$} if for every compact
subset $B$ in $U$, the norms of the second fundamental forms of the
surfaces $\Sigma(n)$ are
uniformly bounded in $B$. } \end{definition}

\section{The proof of Theorem~\ref{th}.} \label{sec:th}

\begin{proof}[Proof of Theorem~\ref{th}] After homothetically scaling, it suffices to prove
Theorem~\ref{th} for $H$-disks $E$, where the radius $R$ of
the related ambient balls is fixed. Henceforth, we will assume that $R=1$.

Arguing by contradiction, suppose that Theorem~\ref{th} fails.
Then there exists a sequence of $H_n$-disks $E(n)$ satisfying the hypotheses of Theorem~\ref{th}
and numbers ${\ve}_n$ going to zero, such that $E(n)\cap \B({\ve}_n)\cap\{x_3>0\}$
contains points $\wt{p}_n $  with
$\lim_{n\to \infty }|A_{E(n)}|(\wt{p}_n )= \infty$.
Since we may assume that $\partial \B(1)$ is transverse
to $E(n)$ then, after replacing $E(n)$ by a subdisk containing
$\wt{p}_n $, we may also assume that
$\partial E(n)\subset\partial \B(1)\cap \{x_3>0\}$. Note that when $H>0$, after such a replacement of $E(n)$ by a subdisk, it
might be the case that $E(n)$ is not
contained in $\B(1)$ or in $\{x_3>0\}$ because the
convex hull property need  not hold.

By the extrinsic curvature
estimates given in Theorem~\ref{ext:cest} for $H$-disks with positive mean curvature,
the mean curvatures $H_n$ of the disks $E(n)$ must be
tending to zero as $n$ goes to infinity. Also, note that for $n$ large, there exist points
$p(n)\in E(n)$ with vertical tangent planes
and with intrinsic distances $d_{E(n)}(p(n),\wt{p}_n)$ converging to zero. Otherwise, after replacing by a subsequence,
small but fixed sized intrinsic balls centered at
the points $\wt{p}_n$ would be  stable
(the inner product of the unit normal to $M$ with the vector $(0,0,1)$ is a
nonzero Jacobi function),
thereby having uniform curvature estimates (see for instance Rosenberg, Souam and Toubiana~\cite{rst1} for
these estimates) and contradicting our choices of the points $\wt{p}_n$
with their norms of the second fundamental form becoming arbitrarily large.

To obtain a contradiction, we are going to analyze the behavior of the  set of points
$\wh{\a}_n$ in $E(n)$ where the tangent planes are vertical; in particular, $\wh{\a}_n$  contains $p(n)$.
We will prove for $n$ large that  $\wh{\a}_n$ contains  a smooth arc
$\a_n$ beginning at $p(n)$ that
moves downward at a much faster rate than
it moves sideways and so $\a_n$ must cross
the  $(x_1,x_2)$-plane near the origin. The existence of such a curve
$\a_n\subset E(n)$ will then contradict the fact
$E(n)$ is disjoint from the  intersection of $\B(1)$ with the $(x_1,x_2)$-plane.

The next proposition  describes the geometry of $E(n)$ around some of its points which are above and close
to the $(x_1,x_2)$-plane and where the tangent planes are vertical. The proposition
states that intrinsically close to such points $E(n)$ must look like
a homothetically scaled {\em vertical} helicoid.
The proof of this result
relies heavily on Theorem~\ref{mainextension} and on the uniqueness
of the helicoid by Meeks and Rosenberg~\cite{mr8}; see also Bernstein and Breiner~\cite{bb1}
for a proof of this uniqueness result.


\begin{proposition} \label{vert} Consider a sequence of points
$q_n\in E(n)\cap C(\frac{1}{2},\frac{1}{2})\cap \{x_3>0\}$ with $x_3(q_n)$
converging to zero and where the tangent planes $T_{q_n}E(n)$ to $E(n)$ are vertical.
Then the numbers $\lambda_n:=|A_{E(n)}|(q_n)$ diverge to infinity and a subsequence
of the surfaces $M(n)=\lambda_n(E(n)-q_n)$ converges on compact subsets of $\rth$
to a vertical helicoid ${\mathcal H}$ containing the $x_3$-axis and with maximal absolute
Gaussian curvature $\frac12$ at
the origin. Furthermore, after replacing by a further subsequence, the multiplicity of the convergence of the surfaces $M(n)$
to ${\mathcal H}$ is one or two.
\end{proposition}
\begin{proof} Crucial to the proof of the proposition is understanding the appropriate
scale to study the geometry of the disks $E(n)$ near $q_n$,
which in principle might not be related to the norms of the second fundamental forms
of  $E(n)$  at the points $q_n$; later we will relate
this new scale to the numbers  $\lambda_n$ appearing in its statement.

Intrinsically near $q_n$, the surface $E(n)$ is   graphical over its tangent
plane $T_{q_n}E(n)$. Recall that each
tangent  plane at $q_n$ is vertical, the sequence of positive numbers
$x_3(q_n)$ is converging to zero and $E(n)$
lies above the $(x_1,x_2)$-plane near $q_n$. It follows that
$B_{E(n)}(q_n,2x_3(q_n))$  cannot be a graph with the norm of its gradient less
than or equal to $1$ over  its (orthogonal) projection to $T_{q_n}E(n)$.

By compactness of $\ov B_{E(n)}(q_n, 2x_3(q_n))$, there is a largest number $r(n)\in (0,2x_3(q_n))$ such that
$B_{E(n)}(q_n,r(n))$ is a graph with the norm of its gradient at most
$1$ over its projection to $T_{q_n}E(n)$. Since $r(n)<2x_3(q_n)$, then
$
\lim_{n\to \infty}r(n)=0$.

 Consider the sequence of
translated and scaled surfaces
\[
\Sigma(n)=\frac{1}{r(n)}(E(n)-q_n).
\]
We claim that   it suffices to prove that a subsequence of
the $\Sigma(n)$ converges with multiplicity
  one or two  to a vertical helicoid containing the $x_3$-axis.
To see this claim holds, suppose
 that a subsequence $\Sigma(n_i)$ of these surfaces converges
 with multiplicity one or two to a vertical
 helicoid ${\mathcal H}'$ containing the $x_3$-axis, then
 $\lambda :=|A_{{\mathcal H}'}|(\vec{0})\in (0,\infty)$ and
 \begin{equation}\label{F1}
\lambda=\lim_{i\to \infty}  |A_{\Sigma(n_i)}|(\vec 0 )
 =\lim_{i\to \infty} r(n_i) |A_{E(n_i)}|(q_{n_i})=\lim_{i\to \infty} r(n_i) \lambda_{n_i}.
 \end{equation}
Since the numbers $r(n_i)$ are converging to zero, equation \eqref{F1} implies that the numbers
$\lambda_{n_i}$ must diverge to infinity,  the sequence of  surfaces
 \[
 M(n_i)=\lambda_{n_i}(E(n_i)-q_{n_i})=\lambda_{n_i}r(n_i)\Sigma (n_i)
 \]
  converges with multiplicity one or two  to ${\mathcal H}=\lambda{\mathcal H}'$ and the
  proposition follows. Thus,
  it suffices to prove that a subsequence of the surfaces $\Sigma(n)$ converges with multiplicity
  one or two  to a vertical helicoid containing the $x_3$-axis.

There are two  cases to consider. \vspace{.1cm}

\noindent{\bf Case~A:} The sequence of surfaces $\Sigma(n)$ has
locally bounded norm of the second fundamental
form in $\rth$. \vspace{.1cm}

\noindent{\bf Case~B:} The sequence of surfaces $\Sigma(n)$ does
{\em not} have locally bounded norm of the second
fundamental form in $\rth$.  \vspace{.1cm}

Suppose that   Case A holds.
A standard compactness argument gives that a
subsequence of the $\Sigma(n)$ converges
$C^{\a}$ to a minimal lamination of $\rth$ for any $\alpha\in (0,1)$: see for
example any of the reference~\cite{bt1,cm23,mr8,sol1}
for these arguments when the surfaces $\S(n)$ are minimal surfaces.
After possibly replacing by a subsequence, we will assume that
the original sequence  $\Sigma(n)$
converges to a minimal lamination ${\mathcal L}$ of $\rth$.

Let $L$ be the leaf of $\mathcal L$ that passes through the
origin and recall that $L$ has a
vertical tangent plane at the origin.
Since $L$ is not a plane, because it is not a graph over this
vertical tangent plane with the norm of its gradient less than 1,
then $L$ is a non-flat leaf of
${\mathcal L}$. By construction $L$  has  bounded norm of the
second fundamental form in compact subsets of $\rth$.
By Theorem~1.6 in~\cite{mr8}, $L$ is a non-limit leaf of
$\mathcal{L}$ and one of the following must hold:
\bit
\item  $L$ is  properly embedded in $\rth$.
\item $L$ is  properly embedded  in an open half-space of $\rth$.
\item $L$ is properly embedded in an open slab of $\rth$.
\eit

Since $L$  is properly embedded in a simply-connected open set
${\mathcal O}$ of $\rth$, then
it separates ${\mathcal O}$ and so it is orientable. Since $L$ is
non-flat and it is complete,
then $L$ is not stable by the classical results of do Carmo, Peng,
Fisher-Colbrie and Schoen~\cite{cp1,fs1}. We claim that the
instability of $L$ implies that the multiplicity
of convergence of domains on $\Sigma(n)$ can be at most one or two. Otherwise,
suppose that the multiplicity of convergence of portions of the surfaces $\S(n) $ to
$L$ is greater than two and let
$\Omega \subset L$ be a smooth compact unstable domain. A
standard argument that we
now sketch  produces a contradiction to the existence of $\Omega$.
By separation properties,
the uniform boundedness of the norms of the second fundamental forms of the compact oriented surfaces
$\S(n)$ in a small $\ve$-neighborhood of $\Omega$ in ${\mathcal O}$
and the properness of $L$, together with the fact that $L$ is not a limit leaf, we have that for $n$ large there exist
three pairwise disjoint compact domains
$\Omega_1(n), \Omega_2(n), \Omega_3(n)$ in $\S(n)$ that are
converging to $\Omega$, each domain is a   normal graph over
$\Omega$ and the unit normal  vectors of $\Omega_1(n)$ and of  $\Omega_3(n)$
at corresponding points
over points of $\Omega$ have positive inner products converging
to 1 as $n$ goes to infinity. Moreover, if we let
$
f_1(n), f_3(n)\colon \Omega\to\R
$
denote the related graphing functions, we can assume that $(f_1(n)-f_3(n))>0$.
 After renormalizing this difference as
 \[
 F(n)=\frac{f_1(n)-f_3(n)}{(f_1(n)-f_3(n))(q)}
 \]
  for some
 $q\in \Int(\Omega)$, elliptic PDE theory implies that a
 subsequence of the $F(n)$ converges to
 a positive Jacobi function on $\Omega$,
 which implies $\Omega$ is stable. This contradiction implies that
 the multiplicity of convergence  is one or two.

Since the multiplicity of convergence of portions of the
$\Sigma(n)$ to $L$ is one or two,
then, for $n$ large, we can lift any simple closed curve
$\gamma $ on the orientable surface $L$ to one or two pairwise
disjoint normal graphs over $\g$ and contained in $\Sigma(n)$, where the number of such lifts
depends on the multiplicity of
the convergence.
This construction gives either one or two simple closed lifted curves in
$\Sigma(n)$. Hence,
since the domains $\Sigma(n)$ have genus zero, it follows
that any pair of transversely intersecting
simple closed curves on $L$ cannot intersect in exactly one
point; therefore, $L$ also has genus zero. By the properness of finite
genus leaves of a minimal
lamination of $\rth$ given by Meeks, Perez and Ros in Theorem 7 in~\cite{mpr3},
$L$ must be properly embedded in $\rth$. In fact, this discussion
demonstrates that  all  the leaves of $\cL$ are proper. Since the leaf
$L$ is not flat, then the strong halfspace theorem by Hoffman and Meeks in~\cite{hm10}
implies that $L$ is the only leaf
in $\cL$.  If $L$ has more than
one end then, by the main result of Choi, Meeks and White in~\cite{cmw1}, $L$ has non-zero flux and so,
by the nature of the convergence of the $\Sigma(n)$ to $L$, the domains $\Sigma(n)$ must have non-zero flux
as well. However, this leads to a contradiction since flux is a
homological invariant and thus $\Sigma(n)$, being
topologically a disk, has zero flux. Therefore $L$ must have
genus zero and one end, which implies that it is
simply-connected. By~\cite{mr8}, $L$ is a helicoid and it remains to
show that it is a vertical helicoid.

\begin{claim}\label{verthelix}
The leaf $L$ is a vertical helicoid.
\end{claim}

\begin{proof}
Recall that $L$ is the limit of
 \[
\Sigma(n)=\frac{1}{r(n)}(E(n)-q_n),
\]
Therefore the helicoid $L$ contains the origin, the tangent plane at
the origin is vertical and, by the definition of $r(n)$, the geodesic ball
centered at the origin is not a graph with norm of its gradient bounded by
one over its tangent plane at the origin.

Let $p$ be the point on the axis of the helicoid $L$ that is closest to
the origin. By the geometric properties of a helicoid and the discussion
in the previous paragraph, it follows that there exist $k_1,k_2>0$ such
that $|p|< k_1$ and that $|A_L|(p)> k_2\in(0,1)$. Let $p_n'\in \Sigma(n)$
be a sequence of points such that $\lim_{n\to\infty}p'_n=p$ and let $p_n\in E(n)$
be the sequence of points such that $p'_n=\frac{p_n-q_n}{r(n)}$. Recall that
$\lim_{n\to\infty}x_3(q_n)=0$ and that $|q_n|^2-x_3(q_n)^2\leq\frac14$. Thus, for $n$ sufficiently large the following holds:
\begin{itemize}
\item $|p_n-q_n|\leq r(n)|p_n'|\leq 2k_1r (n)$;
\item $|p_n|\leq |p_n-q_n|+|q_n|\leq 2k_1r(n)+\sqrt{\frac14+x^2_3(q_n)}$;
\item
$|x_3(p_n)|\leq |x_3(p_n-q_n)|+|x_3(q_n)|\leq |p_n-q_n|+x_3(q_n)\leq 2k_1r(n)+x_3(q_n)$.
\end{itemize}
In particular, on the original scale, the points $p_n$ and $q_n$ have the
same limit. 

The new sequence of surfaces
\[
\Sigma'(n)=\frac{1}{r(n)}(E(n)-p_n)
\]
converges to the translated helicoid $\widehat L=L-p$ whose axis contains
the origin. Clearly, in order to prove the claim, it suffices to show that
$\widehat L$ is a vertical helicoid. Suppose that the axis of $\widehat L$
makes an angle $\theta>0$ with the $x_3$-axis.
Let $\tau>0$ and let $G_\tau:=G(\frac 14, \tau,\frac 1{20} )$ as given by
Theorem~\ref{mainextension} with $\ve=\frac14$ and $\ov\ve=\frac1{20}$.
Let $z_n'\in \Sigma'(n)\cap \B(\frac{G_\tau}{20k_2})$ be a sequence of points where the maximum of the functions
\[
|A_{\Sigma'(n)\cap \B(\frac{G_\tau}{20k_2})}|(\cdot)|\frac{G_\tau}{20k_2}-|\cdot||
\]
is obtained. Clearly, $\lim_{n\to\infty}|z_n'|=0$. Let $z_n\in E_n$ be the
sequence of points such that $z_n'=\frac{1}{ r(n)}(z_n-p_n)$. By the previous
discussion,  $\frac{1}{r(n)}(E(n)-z_n)$ converges to the helicoid $\widehat L$.

Consider the sequence of surfaces
\[
M_\tau(n)=E(n)-p_n=r(n)\Sigma'(n).
\]
Recall that $\partial E(n)\subset\partial \B(1)\cap \{x_3>0\}$,
$E(n)\cap \{x_3=0\}=\O$, $|p_n|\leq 2k_1r(n)+\sqrt{\frac14+x^2_3(q_n)}$ and $|x_3(p_n)|\leq 2k_1r(n)+x_3(q_n)$.
Therefore, when $n$ is sufficiently large and abusing the notation, we
can let $M_\tau(n)$ denote the subdisk of $M_\tau(n)$ containing the component
of $M_\tau(n)\cap \B(\frac14)$ that contains the origin and with boundary in
$\partial \B(\frac 14)$.  Recall that the constant mean curvatures $H_n$ of $E(n)$, and thus of $M_\tau(n)$, are going to zero as $n$ goes to infinity. Therefore, when $n$ is sufficiently large, we have that $H_n\in(0,\frac 18)$. Note that 
\[
M_\tau(n)\cap \{x_3=-2k_1r_n- x_3(q_n)\}=\O.
\]
Let $\eta_n= \frac{G_\tau r(n)}{k_2}$. Then, $\lim_{n\to\infty }\eta_n=0$
and for $n$ sufficiently large, the sequence $M_\tau(n)$ satisfies the following properties.
\begin{itemize}
\item $|A_{M_\tau(n)}|(\vec 0)>\frac{k_2}{r(n)}=\frac{G_\tau}{\eta_n}$;
\item the maximum of the function
$|A_{M_\tau(n)\cap \B(\frac{\eta_n}{20})}|(\cdot)|\frac{\eta_n}{20} -|\cdot||$
is obtained at the point $z_n$;
\item on the scale of the norm of the second fundamental
form $M_\tau(n)-z_n$ looks like the helicoid $\widehat L$.
\end{itemize}

Let $\Pi$ denote the plane containing $z_n$ and perpendicular to the
axis of $\widehat L$. Then, under these hypotheses, Theorem~\ref{mainextension}
implies that there exist constants $\Omega(\tau)$ and $\omega(\tau)$
such that $M_\tau(n)$ contains a 3-valued graph over the annulus in
$\Pi$ centered at $z_n$ of outer radius $\frac1{4\Omega(\tau)}$ and
inner radius $\frac{\omega(\tau)}{|A_{M_\tau(n)}|(\vec 0)}$. Note that
$\lim_{n\to\infty} z_n=\vec 0$ and that $\lim_{n\to\infty}\frac{\omega(\tau)}{|A_{M_\tau(n)}|(\vec 0)}=0$. Therefore, by taking $\tau$ sufficiently
small, depending on $\theta$, and $n$ sufficiently large, this geometric
description contradicts the fact that $M_\tau(n)\cap \{x_3=-2k_1r(n)+ -x_3(q_n)\}=\O$. This contradiction finishes the proof of the claim.
\end{proof}

This finishes the proof of the proposition when Case A holds.
To complete the proof of Proposition~\ref{vert}, it suffices
to demonstrate the following assertion, which is the difficult
point in the proof of Proposition~\ref{vert}.

\begin{assertion} \label{case b} Case B does not occur.
\end{assertion}
\begin{proof}
Some of the techniques used
to eliminate Case B
are motivated by techniques and results developed
in~\cite{cm26,mpr10,mpr11} and their corresponding proofs.
Arguing by contradiction, assume that the sequence of surfaces $\Sigma(n)$
does not have locally bounded norm of the second fundamental
form in $\rth$.

For
clarity of  exposition, we will replace the sequence of surfaces
$\Sigma(n)$ by a specific
subsequence such that for some non-empty closed set
${\chi}$ in $\rth$ with $\chi$ different from $\rth$,
the new sequence of surfaces has locally bounded
norm of the second fundamental form in
$\rth-\chi$, converges $C^{\a}$, for any $\alpha\in (0,1)$, to a non-empty minimal
lamination ${\mathcal L}_{\chi}$ of $\rth -\chi$ and
no further subsequence has locally bounded norm of the
second fundamental form in  $\rth-\chi'$, where
$\chi'$ is a proper closed subset of $\chi$. This reduction
is explained in the next claim.
\begin{claim} \label{chi} After replacing by a subsequence,
the sequence of surfaces $\Sigma(n)$
satisfies the following properties:
\ben[1.]
\item There exists a closed non-empty set $\chi\subset \rth$
such that for every point $s\in\chi$ and for each $k\in \N$,
there exists an $N(s,k)\in \N$ such that for each
$j\geq N(s,k)$, there exists a point
$p(j)\in \Sigma(j)\cap\B(s,\frac{1}{k})$ with $|A_{\Sigma(j)}|(p(j))\geq k$.
\item The sequence $\Sigma(n)$ has locally bounded
norm of the second fundamental form in $\rth-\chi$.

\item The sequence $\Sigma(n)$ converges $C^{\a}$,
for any $\alpha\in (0,1)$, on
compact domains of $\rth-\chi$ to a non-empty minimal
lamination ${\mathcal L}_{\chi}$ of $\rth-\chi$.

\item There exists a maximal  open horizontal slab or open
half-space $W$ in $\rth-\chi$, $\vec{0}\in W$ and
${\mathcal L}={\mathcal L}_{\chi}\cap W$ is a
non-empty minimal lamination of  $W$.

\item The leaf $L$ of ${\mathcal L} $ that passes through the
origin is non-flat and contains an
intrinsic open disk $\Omega$
passing through the origin
which is the limit of the surfaces $B_{\Sigma(n)}(\vec{0},1)$ and
$\Omega$ is a graph over its projection
to $T_{\vec{0}}$ $\Omega$ which is a vertical plane, and with
norm of the gradient of the graphing function at most one.
\een
\end{claim}

\begin{proof}[Proof of  Claim~\ref{chi}.] We begin by constructing
the set $\chi$ and the related subsequence
of surfaces $\Sigma(n)$ described in the claim.
The assumption that the original sequence
of surfaces  does not have locally bounded norm of the second
fundamental form in $\rth$ implies there
exists a point $q(1)\in \rth$ such that, after replacing this
sequence of surfaces by a subsequence
$\Gamma_1 :=\{
\Sigma(1,n)\}_{n\in \N}$, the surfaces in $\Gamma_1$ satisfy the following property:
For each $k\in \N$, there is a point
$p(1,k)\in \Sigma(1,k)\cap \B(q(1),\frac1k )$ with $|A_{\Sigma(1,k)}|(p(1,k))\geq k$.

Let $\Q$,  $\Q^+$ denote the set of rational numbers and  the
subset of positive rational numbers, respectively.
Consider the countable collection of balls
$${\mathcal B}=\{\B(x,q)\mid x\in \Q^3, \,q\in\Q^+\},$$ and let
${\mathcal B}=\{B_1,B_2,\ldots,B_n,\ldots\}$ be an ordered
listing of the elements in ${\mathcal B}$ where
$q(1)\in B_1$. If $\G_1$ has locally bounded norm of the second
fundamental form  in $\rth-\{q(1)\}$, then
$\chi=\{q(1)\}$ and we stop our construction of the set $\chi$.
Assume now that $\G_1$ does not have locally
bounded norm of the second fundamental form
in $\rth-\{q(1)\}$. Let $B_{n(2)}$ be the first indexed ball
in the ordered listing of ${\mathcal B}-\{B_1\}$, such
that the following happens: there is a point $q(2)\in B_{n(2)}-\{q(1)\}$, a subsequence
$\Gamma_2:=\{\Sigma(2,1),\Sigma(2,2),\ldots, \Sigma(2,k),\ldots\}$ of $\G_1$
together with points $p(2,k)\in \Sigma(2,k)\cap \B(q(2),\frac{1}{k})$
where $|A_{\Sigma(2,k)}| (p(2,k))\geq k$.  Note that
$B_{n(2)}$ is just the first ball in the list  ${\mathcal B}-\{B_1\}$
that contains a point $q$ different from $q(1)$
such that the norms of the second fundamental forms of the surfaces
in the sequence $\G_1$ are not bounded in any
neighborhood of $q$, and after choosing such a point, we label it as $q(2)$.

If $\G_2$ does not have locally bounded norm of the second
fundamental form in $\rth-\{q(1),q(2)\},$ then there
exists a first ball $B_{n(3)}$ in the ordered list
${\mathcal B}-\{B_1, B_{n(2)}\}$   such that there is a point
$q(3)\in [B_{n(3)}-\{q(1),q(2)\}]$ and such that after replacing
$\G_2$ by a subsequence
$\G_3:=\{\Sigma(3,1),\ldots, \Sigma(3,k),\ldots\}$, there are
points $p(3,k)\in \Sigma(3,k)\cap \B(q(3),\frac{1}{k})$
with $|A_{\Sigma(3,k)}| (p(3,k))\geq k$.

Define $n(1)=1$. Then, continuing the above construction
inductively, we obtain at the $m$-th stage a
subsequence
$\G_m=\{\Sigma(m,1),\ldots,\Sigma(m,k),\ldots\}$, $\G_m\subset \G_{m-1}$, a set of distinct
points $\{q(1),q(2),\ldots, q(m)\}$
and related balls $B_{n(1)}$, $B_{n(2)},\ldots, B_{n(m)}$ in ${\mathcal B}$
such that for each $i\in\{1,\ldots,m\}$, each
$\Sigma(m,k)$   contains points $p(m,k,i)$ in the balls $\B(q(i),\frac{1}{k})$,
where the norm of the second fundamental
form of $\Sigma(m,k)$ is at least $k$.  Note that $n(i+1)>n(i)$ for all $i$.

Let $\G_{\infty}=\{\Sigma(1,1),\Sigma(2,2),\ldots, \Sigma(n,n),\ldots\}$
be the related diagonal sequence and
let $\chi$ be the closure of the set $\{q(n)\}_{n\in\N}$.
By  the definitions of $\chi$ and of $\G_{\infty}$, for each
point $s\in \chi$ and for each $k\in \N$, there
exists $N(s,k)\in \N$ such that for $j\geq N(s,k)$, there are points
$p(j)\in \Sigma(j,j)\cap\B(s,\frac{1}{k})$ with
$|A_{\Sigma(j,j)}| (p(j))\geq k$.

We claim that the sequence $\{\Sigma(n,n)\}_{n\in \N}=\G_{\infty}$
has locally bounded norm of the second fundamental
form in $\rth-\chi$. To prove this, it suffices to check that
given a point $y\in \rth -\chi$, there exists a closed ball
$\overline{B}$ such that $y \in \Int(\overline{B})$ and
$\G_{\infty}$ has uniformly bounded norm of the second fundamental
form in $\overline{B}$. Choose an
$r\in (0,\frac{1}{4}d_{\rth}(y,\chi))\cap \Q^+$ and let $x\in\Q^3$
be a point of distance less
than $r$ from $y$. By the triangle inequality, one has
$\B(x,2r)\subset \rth-\chi$ and suppose $\B(x,2r)=B_J\in \cB$, for
some $J\in \N$. We claim that $\G_{\infty}$ has uniformly bounded norm of the second
fundamental form on $\overline{\B(y,r)}\subset B_J$.
Otherwise, by compactness of $\overline{\B(y,r)}$, there
exists a point $q\in \overline{\B(y,r)}\subset B_J$ such that the
norms of the second fundamental forms of the surfaces in the
sequence $\G_\infty$ are not bounded in any neighborhood of $q$.  By
the inductive construction then
$\{q(1),\ldots,q(J)\}\cap B_J\neq \O$; however, $B_J\subset [\rth -\chi]$
and $\{q(1),\ldots,q(J)\}\subset \chi$.
This contradiction implies that $\G_{\infty}$ has locally
bounded norm of the second fundamental form in $\rth-\chi$.

Now replace $\G_{\infty}$ by a subsequence, which  after
relabeling we denote by $\G=\Sigma(n)$,
such that the surfaces $\Sigma(n)$ converge $C^{\a}$ to a
minimal lamination ${\mathcal L}_{\chi}$ of $\rth-\chi$ for any $\alpha\in (0,1)$.
This completes the proofs of the first three items of the claim.

Recall that the intrinsic open balls
$B_{\Sigma(n)}(\vec{0}, 1)\subset \Sigma(n)$ are graphs of
functions with the norms of their gradients at most
one over their vertical tangent planes.
Hence, after refining the subsequence further, assume that the graphical
surfaces  $B_{\Sigma(n)}(\vec{0},1)$ converge smoothly to a
graph $\Omega$  over a vertical plane. In particular for
some $\delta\in (0,\frac14)$ small, depending on curvature
estimates for $H$-graphs, we can find arcs
\[
\g(n)\subset B_{\Sigma(n)}(\vec{0},1)\cap \{|x_3|\leq \delta\}
\] with one end point in the  plane $\{x_3=\delta\}$ and the other end point
in $\{x_3=-\delta\}$ and the $\delta$-neighborhood of $\g(n)$ in $\Sigma(n)$ is
contained in $B_{\Sigma(n)}(\vec{0},\frac{1}{2})$.
In particular, the curves  $\g(n)$ stay an intrinsic distance  at least $\delta$
from any points of $\Sigma(n)$ with very large curvature.

\begin{claim}\label{horizontalplane}
Let $p\in\chi$. Then, after choosing a subsequence, there is a
horizontal plane $Q(p)$ passing through $p$ such that
$Q(p)-\{p\}$ is the limit of a sequence
of  3-valued graphs $G(n)\subset \Sigma(n)-{p}$
over the annulus $A(n,\frac1n)$ with the norms of their gradients less than $\frac1n$.
\end{claim}
\begin{proof}
 Recall that
\[
\Sigma(n)=\frac{1}{r(n)}(E(n)-q_n),
\]
and as $n$ goes to infinity,  the constant mean curvatures of the surfaces
$\Sigma (n)$ are converging to zero, while the distances from $\partial \Sigma (n)$
to the origin are diverging to infinity.

By item 1 of Claim~\ref{chi}, for each $k\in \N$,
there exists an $N(k)\in \N$ such that for each $n\geq N(k)$, there exists a point
$p_k(n)\in \Sigma(n)\cap\B(p,\frac{1}{k})$ with $|A_{\Sigma(n)}|(p_k(n))\geq k$.
Let $\widetilde \Sigma(n)=\Sigma(n)-p_k(n)$ and note that
$|A_{\widetilde \Sigma(n)}|(\vec{0})\geq k$.
Without loss of generality,  we may assume that as $n$ goes to infinity,
$\partial \widetilde \Sigma(n)\subset \partial \B(R(n))$
with $\lim_{n\to \infty}R(n)=\infty$ and that its mean curvature $H(n)\in (0,\frac1{2R(n)})$.

Given $i\in \mathbb N$, let $\tau_i$ and $\ov \ve_i$ be two sequences of
positive numbers going to zero as $i$ goes to infinity and let $\ve_i$ be
a sequence of positive numbers going to infinity such that $\ve_i\slash \Omega_{\tau_i}$
is also going to infinity, where $\Omega_{\tau_i}$ is given by Theorem~\ref{mainextension}.
By the discussion at the beginning of the proof and after possibly considering $H(n)$-subdisks
of $\widetilde\Sigma(n) $ that, abusing the notation, we still call $\widetilde \Sigma(n)$,
the following holds: for each $i\in \mathbb N$, there exists $n_i\in \mathbb{N}$ such
that $|A_{\widetilde \Sigma(n_i)}|(\vec{0})\geq \max(i\omega_{\tau_i},G_{\tau_i})$,
$\partial \widetilde \Sigma(n_i)\subset \partial \B(\ve_i)$ and $H(n_i)\in(0,\frac1{2\ve_i})$,
where $\omega_{\tau_i}$ and $G_{\tau_i}$ are again obtained by applying Theorem~\ref{mainextension}.

By Theorem~\ref{mainextension} with $\eta=1$ there exist points
$p(n_i)\in \ov{\B}(\ov{\ve}_i)$ such that after translating
$\widetilde \Sigma(n_i)$ by $-p(n_i)$, the following geometric
description of $\widetilde \Sigma(n_i)$ holds:

On the scale of the norm of the second
fundamental form  $\widetilde \Sigma(n_i)$ looks like one or two helicoids nearby  the
origin and, after a rotation  that turns these helicoids into
vertical helicoids, $\widetilde \Sigma(n_i)$  contains a 3-valued graph
$u(n_i)$ over  $A(\ve_i\slash\Omega_{\tau_i},\frac{1}{i})$
with the norm of its gradient less than $\tau_i$ and with its inner boundary in
$\B(\frac{10}{i})$. Since, as $i$ goes to infinity, $\ve_i\slash\Omega_{\tau_i}$
goes to infinity, and $\tau_i$ and $\ov\ve_i$ are going to zero, the sequence of
3-valued graphs $u(n_i)$ is converging to the $(x_1,x_2)$-plane.

Thus, after possibly reindexing the elements of the sequence of surfaces $\Sigma(n)$,
choosing a subsequence and applying a rotation $\mathcal R$, the previous
discussion shows that there is a
horizontal plane $Q(p)$ passing through $p$ such that
$Q(p)-\{p\}$ is the limit of a sequence
of  3-valued graphs $G(n)\subset \Sigma(n)-{p}$
over the annulus $A(n,\frac1n)$ with the norms of their gradients less than $\frac1n$.
In order to finish the proof of the claim, it suffices to show that the rotation
$\mathcal R$ is in fact the identity, namely that the helicoids forming on the scale
of the norm of the second fundamental form are vertical. This follows using arguments
that are similar to the ones used in the proof of Claim~\ref{verthelix}. This observation finishes the proof of the claim.
\end{proof}

We now continue with the proofs of items 3 and 4 in Claim~\ref{chi}.
The previous claim gives  that in the slab
$S=\{-\frac{\delta}{2}< x_3<\frac{\delta}{2}\}$, the sequence of
surfaces $\Sigma(n)\cap S$ has locally bounded norm of the second fundamental
form. Otherwise there
would be a point $p \in \chi\cap S$ and  a horizontal plane
$Q(p)$ passing through $p$ such that $Q(p)-\{p\}$ is the limit
of 3-valued   graphs $G(n)\subset \Sigma(n)$. This is impossible
since these 3-valued graphs would intersect
the arcs $\gamma(n)$
contradicting embeddedness of the surfaces $\S(n)$.

Let ${\mathcal T}$  be the set of open horizontal slabs
containing the origin and disjoint from $\chi$. The set $\mathcal{T}$ is
non-empty as it contains $S$.  Since $\chi\neq  \O$, the union
$W=\cup_{\widehat{S}\in {\mathcal T}}\widehat{S}$ is the largest open slab
in ${\mathcal T}$  or the largest open half-space containing $S$
for which $W\cap\chi=\O$ and containing the origin. The fact that
${\mathcal L}={\mathcal L}_{\chi}\cap W$ is a non-empty lamination
of  $W$ is clear by definition of lamination and this proves item 4. The validity of
item 5 is also clear. This completes the proof of Claim~\ref{chi}.
\end{proof}

\begin{claim} \label{claim:incomplete}
The leaf $L$ given in item 5 of Claim~\ref{chi} has at most one  point of
incompleteness in each component of $\partial W$ and any such
point lies in $\chi$.  Here we refer to a point $p \in \rth$
as being a point of {\em incompleteness}
of $L$, if $p$ is the limiting end point in $\rth$ of some
proper arc $\a\colon [0,1) \to L \,$ of finite length.
\end{claim}
\begin{proof}
Since $L$ is a leaf of a lamination of  $W$, a point of
incompleteness of $L$ must lie in $\partial W$.
Let $P$ be one of the  horizontal planes in $\partial W$
and let $y_1\in P $. Suppose
that there exists some smooth,  proper, finite length  arc
$\a_{y_{1}}\colon [0,1)\to L$  with limiting end point $y_1$
and with beginning point $\a_{y_1}(0)$ in $L$; then such a
point $y_1$ is a point of incompleteness of the leaf $L$.
Note that the half-open arc $\a_{y_1}$ can be taken to be a $C^1$-limit of smooth
embedded compact arcs $\a_{y_1}^n\subset [\Sigma(n)\cap W]$ of lengths
converging to the length of $\a_{y_1}$. Let $y_1(n)\in \a_{y_1}^n$
be the sequence of end points of $\a_{y_1}^n$, which are
converging to $y_1$.

We claim that there must be positive
numbers $\tau_n$ with $\lim_{n\to\infty }\tau_n=0$, such that
$B_{\Sigma(n)}({y_1}(n),\tau_n)$ contains points of arbitrarily
large curvature as $n$ goes to $\infty$. In particular,
it would then follow that ${y_1}\in \chi$ and that we could
modify the choice of the curves $\a_{y_1}^n$ so that the
endpoints $y_1(n)$ are points with arbitrarily large norm of the second
fundamental form as $n$ goes to infinity. To see that this claim
holds, suppose to the contrary that after choosing a subsequence,
for some small fixed $\ve>0$, the norms of the second
fundamental forms of the intrinsic balls $B_{\Sigma(n)}(y_1(n),\ve)$ are bounded, and
by choosing $\ve$ smaller, we may assume that these
intrinsic balls are disks that are graphs  over their projections
to their tangent planes at $y_1(n)$ with the norms of their gradients at most 1.
Then a subsequence of these intrinsic balls
would converge to an open minimal disk $D$ with $y_1\in D$ and $D\cap L$ contains
a subarc of $\a_{y_1}$. By the maximum principle, $D$ contains points on both sides
of $P$. Since the distance from $P$ to
$\chi$ is zero, then Claim~\ref{horizontalplane} implies that $P$ is also the limit
of a sequence of 3-valued graphs $G(n)\subset \Sigma(n)$.
But then for $n$ large, this sequence $G(n)$ must intersect $D$ transversally at some
point $q_n\in D$.  Since the disks
 $B_{\Sigma(n)}(y_1(n),\ve)$ converge smoothly to $D$, then for $n$ sufficiently large,
 $G(n)$ also intersects  $B_{\Sigma(n)}(y_1(n),\ve)$
 transversely at some point.   This contradiction then proves the claim that
 ${y_1}\in \chi$ and that we could modify the choice
 of the curves $\a_{y_1}^n$ so that the
endpoints $y_1(n)$ are points with arbitrarily large norm
of the second fundamental form as $n$ goes to infinity.

 Without loss of generality, assume that $W$ lies below $P$.
Suppose that there is a point ${y_2}\in P$, ${y_2}\not=y_1$,
of incompleteness for $L$ with related proper arc
$\a_{{y_2}}$ beginning at $\a_{y_2}(0)$ with limiting end point
${y_2}$, and related approximating curves
$\a_{{y_2}}^n$ in $\Sigma(n)$ with  end points ${y_2}(n)$ converging to ${y_2}$,
where the norm of the second fundamental form is
arbitrarily large. Also assume that $\a_{y_1},\a_{y_2}$ are
sufficiently close to their different limiting
end points so that, after possibly replacing them by subarcs,
they are contained in closed balls in $\rth$
that are a positive distance from
each other. By Claim~\ref{horizontalplane} there exist sequences
of 3-valued graphs
$G_1(n)$, $G_2(n)$ in $\Sigma(n)$ with inner boundary curves
converging to the points $y_1(n)$, ${y_2}(n)$,
respectively, and these 3-valued graphs can be
chosen so that each of them collapses to the
punctured plane $P$ punctured at $y_1,y_2$, respectively. For $n$ large,  $G_1(n)$
must lie ``above'' the point $\a_{y_2}^n(0)$ which implies that
near $y_2(n)$, $G_1(n)$ must lie above $G_2(n)$ as well,
otherwise, $G_1(n)$ would intersect the arc $\a_{y_2}^n$, which
would contradict the embeddedness of $\S(n)$. Reversing the roles
of $G_1(n)$ and $G_2(n)$, we find that $G_2(n)$ must lie above $G_1(n)$ near $y_1(n)$;
hence, for $n$ sufficiently large, the multigraph $G_1(n)$ must intersect $G_2(n)$.
This contradicts that $\S(n)$ is embedded and thereby proves that there is at most one point
$y\in P\cap \chi$ that is a point of incompleteness of $L$.
\end{proof}

So far we have shown that $L$ can have at most one point of incompleteness on each of the
components of $\partial W$ and all the points of incompleteness are contained
in $\partial W$. Define {\bf S} to be the set of points of incompleteness of $L$.
By the previous claim, the possibly empty set ${\bf S}$ has at most two points.
The next claim describes some consequences of this finiteness of {\bf S}.

\begin{claim} \label{claim3.5}
The leaf $L$ has genus zero and the closure $\ov{L}$ of $L$ in
$\rth- {\bf S}$ is a minimal
lamination  of $\rth -{\bf S}$ that is contained in $\ov{W}-{\bf S}$.
\end{claim}

\begin{proof}
We first prove that $L$ is not a limit leaf of ${\mathcal L}$.
If $L$ were a limit leaf of ${\mathcal L}$,
then its oriented 1 or 2-sheeted cover  would be stable
by Theorem~4.3 in~\cite{mpr19} or Theorem~1 in~\cite{mpr18} by Meeks, Perez and Ros.
Stability gives curvature estimates on $L$ away from ${\bf S}$
and thus the closure $\overline{L}$ of $L$ in
$\rth -{\bf S}$,  is a minimal lamination of $\rth -{\bf S}$. By
Corollary~7.2 in~\cite{mpr10} by Meeks, Perez and Ros, the closure of $\overline{L}$
in $\rth$ is a minimal lamination of $\rth$ and
each leaf is stable; in particular, by stability the closure of $\overline{L}$ in
$\rth$ must be a plane.
This is a contradiction because $L$ lies in a horizontal slab
and its tangent plane at the origin is vertical.

Let $\wt W\subset W$ be the smallest open slab or half-space containing $L$. We claim that $L$
is properly embedded in  $\wt W$. If not, then the closure $\ov L$ of $L$ in $\wt W$ is
a sublamination of $\cL$ with a limit leaf $T$. Note that since $\wt W$ is the
smallest open slab or half-space in $W$ containing $L$, $T$ cannot be a plane. The proof of
Claim~\ref{claim:incomplete} applies to $T$ to show that it has
at most one point of incompleteness on each component of $\partial W$
and any such point lies in $\chi$. By the arguments in the previous
paragraph, $T$ must be a plane. This is a contradiction and therefore $L$
is properly embedded in  $\wt W$.

Since $L$  is
properly embedded in a simply-connected open set
of $\rth$, then it separates the open set and so it is orientable.
As $L$ is not stable, the convergence
of portions of $\Sigma(n)$ to $L$ has multiplicity one or two
which implies that $L$ has genus zero and zero
flux; see the discussion when Case A holds for further details
on this multiplicity of convergence bound and the
genus-zero and zero-flux properties.

It remains to prove that the closure  $\ov{L}$ of $L$ in
$\rth- {\bf S}$ is a minimal lamination  of $\rth -{\bf S}$.
By the minimal lamination closure theorem by Meeks and Rosenberg in~\cite{mr13} (specifically see
Remark 2 at the end of this paper),
this desired result is equivalent to proving that the
injectivity radius function of $L$ is bounded away from zero
on compact subsets of $\rth-{\bf S}$.
Otherwise,
the local picture theorem on the scale of
the topology by Meeks, Perez and Ros in~\cite{mpr14} implies that there exists a sequence of
compact domains $\Delta_n\subset L$ such that related homothetically scaled and translated domains
$\wt{\Delta}_n$ converge
smoothly with multiplicity one to a  properly embedded genus-zero
minimal surface $L_\infty$ in $\rth$ with injectivity
radius one or there exist closed geodesics
$\gamma_n \subset \wt{\Delta}_n$ with nontrivial flux.
So, the latter cannot happen as $L$ has zero flux. On the other hand, if the
former happens then,  by the classification of non-simply-connected, properly embedded minimal
planar domains in $\rth$ in \cite{mpr6} by Meeks, Perez and Ros, the surface
$L_\infty$ is a catenoid or a Riemann minimal example and so has non-zero flux; therefore
$L_\infty$  cannot be the limit with multiplicity one of translated and scaled domains in
$L$. In either case, we have obtained a contradiction, which proves
that the injectivity radius function of $L$ is bounded away
from zero on compact subsets of $\rth- {\bf S}$.
This completes the proof of the claim.
\end{proof}

Finally we prove that Case~B cannot occur. By Claim~\ref{claim3.5},
the closure $\ov{L}$ of $L$ in $\rth- {\bf S}$
is a minimal lamination; in the technical language developed in~\cite{mpr11} by Meeks, Perez and Ros,
$\ov{L}$ is a singular minimal lamination of $\rth$  with a countable
(in fact at most two) set of singular points and its regular part
in  $\rth- {\bf S}$ contains
the genus-zero leaf $L$.
By item~6 of Theorem~1.8 in~\cite{mpr11},
the closure of $L$ in $\rth$ must be a properly embedded
minimal surface in $\rth$. Since $L$ is non-flat and is
contained in the open slab or half-space  $W$,  the half-space
theorem in~\cite{hm10} applied to $\ov{L}$ gives a contradiction, which  proves that
Case~B cannot occur. This completes the proof of Assertion~\ref{case b}. \end{proof}

Since Case B does not occur, then Case~A must occur. We have already proven
Proposition~\ref{vert} when Case A holds, thus,
the proof of Proposition~\ref{vert} is finished.
\end{proof}

 The next corollary follows immediately from Proposition~\ref{vert}.

\begin{corollary}\label{lookslikehel}  Given $\ve_1\in (0,1/2)$, there exist $\ve_2>0$
such that the following holds. Let $E$ be an $H$-disk satisfying $$E\cap \B(1)\cap\{x_3=0\}
=\O \quad \mbox{and} \quad \partial E\cap \B(1)\cap\{x_3>0\}=\O.$$
Suppose
$p\in E\cap \B(\ve_2)\cap\{x_3>0\}$ where the
tangent plane to $E$ is vertical. Then
there exists a vertical helicoid ${\mathcal H}$ with maximal
absolute Gaussian curvature
$\frac12$ at the origin such that the connected component of
$|A_{E}|(p)[E-p]\cap \B(1)$
containing the origin is a normal graph $u$ over its projection
$\Omega\subset[\B(1+2\ve_1)\cap \mathcal H]$,
 where $\Omega\supset[\B(1-2\ve_1)\cap \mathcal H]$ and
 $\|u\|_{C^2}\leq \ve_1$.
Furthermore, $|A_{E}|(p)[E-p]\cap \B(1)$
consists of 1 or 2 components and if there are two components, then the component
not passing through the origin is a normal graph $u'$
over its projection $\Omega'\subset[\B(1+2\ve_1)\cap \mathcal H]$,
 where $\Omega'\supset[\B(1-2\ve_1)\cap \mathcal H]$ and $\|u'\|_{C^2}\leq \ve_1$.
\end{corollary}

As a consequence of Proposition~\ref{vert} and
Corollary~\ref{lookslikehel}, we obtain the next claim.

\begin{claim} \label{no_vert} There exists an $\ve>0$ such
that the following holds. If $E$ is an $H$-disk satisfying
the hypotheses of Theorem~\ref{th} for $R=1$, then $E\cap \B(\ve)\cap \{x_3>0\}$
contains no vertical tangent planes.
\end{claim}
\begin{proof}
Arguing by contradiction, let $E(n)$ be a sequence of
$H_n$-disks satisfying the hypothesis of
Theorem~\ref{th}  for $R=1$, together with a sequence of points $p(n)$
in $E(n)\cap \B(\frac{1}{n})\cap \{x_3>0\}$ with vertical
tangent planes. Given $\delta>0$,  let $\Gamma_n(\de)$
be the connected component of
$$E(n)\cap \B(\delta)\cap \{x_3>0\}\cap N_n^{-1}(\{x_1^2+x_2^2=1\})$$
containing $p(n)$, where $N_n$ is  the Gauss
map of $E(n)$. It follows from Corollary~\ref{lookslikehel}
that given $\rho\in (0, 1/2)$, there exists
$\delta>0$ such that for $n$ large, $\Gamma_n(\de)$ is a possibly disconnected
analytic curve that can be parameterized to have unit
speed and so that $|\langle \dot \Gamma_n(\de),(0,0,-1)\rangle|>1-\rho$. By taking
$\rho$ sufficiently small, and for $n$ sufficiently large, this curve
crosses the $(x_1,x_2)$-plane nearby the origin.
This is impossible because the
disks $E(n)$ are disjoint from $\B(1)\cap \{x_3=0\}$. This
contradiction completes the proof of the claim.
\end{proof}

To complete the proof of Theorem~\ref{th}, recall that at the
beginning of the proof of Theorem~\ref{th}, we showed
that if the theorem fails, then there exists a sequence of
$H_n$-disks $E(n)$ with $\partial E(n)\subset \partial \B(1)$ and
$H_n\leq 1$ satisfying the hypotheses
of Theorem~\ref{th} and points $p_k(n)\in E(n)$ with vertical
tangent plane and converging to the origin.  This contradicts
Claim~\ref{no_vert} and the proof of  Theorem~\ref{th} is completed.
\end{proof}

The next corollary follows immediately from   Theorem~\ref{th}
by a simple rescaling argument. It roughly states that
we can replace the  $(x_1,x_2)$-plane   by any surface that has a fixed uniform
estimate on the norm of its second fundamental form.

\begin{corollary} \label{cest2}
Given $a\geq 0$,
there exist $\ve\in(0,\frac{1}{2})$ and $C_{a} >0$ such that
for any $R>0$, the following holds.
Let $\Delta$ be a compact immersed surface in $\B(R)$ with
$\partial \Delta \subset \partial \B(R)$, $\vec{0}\in \Delta$
and satisfying $|A_{\Delta}| \leq a/R$. Let $\cD$ be an $H$-disk such that
$$\cD\cap \B(R)\cap\Delta=\O \quad \mbox{and} \quad \partial \cD\cap \B(R)=\O.$$
Then:
\begin{equation} \label{eq1*}
\sup _{x\in \cD\cap \B(\ve R)} |A_{\cD}|(x)\leq \frac{C_{a}}{R}.
\end{equation} In particular, if $\cD\cap \B(\ve R)\neq \O$,
then $H\leq \frac{C_{a}}{R}$.
\end{corollary}

\section{Consequences of the one-sided curvature estimate.} \label{sec:consequences}

In this section we state several theorems that depend on  the
one-sided curvature estimate
in Theorem~\ref{th}. We begin by making the following definition.

\begin{definition} Given a point $p$ on a surface
$\Sigma\subset \rth$, $\S (p,R)$ denotes the closure of
 the component of $\Sigma \cap {\B}(p,R)$ passing through $p$.
\end{definition}

In~\cite{mt8}, we apply the one-sided curvature estimate in
Theorem~\ref{th} to prove a relation between intrinsic
and extrinsic distances in an $H$-disk, which can be
viewed as a {\em weak chord arc} property.
This result was motivated by and generalizes Proposition~1.1 in~\cite{cm35} by
Colding and Minicozzi. More precisely, the statement is the following.

\begin{theorem}[Weak chord arc property, Theorem 1.2
in~\cite{mt8}] \label{thm1.1} There exists a $\delta_1 \in (0,
\frac{1}{2})$  such that the following holds.

Let $\S$ be an   $H$-disk in $\rth.$  Then for all
intrinsic closed balls $\ov{B}_\S(x,R)$ in $\S-
\partial \S$:

\ben \item $\S (x,\delta_1 R)$ is a disk with piecewise smooth boundary
$\partial \Sigma(\vec{0},\delta_1 R)\subset \partial \B(\de_1R)$. \item
$
 \S (x, \delta_1 R) \subset B_\S (x, \frac{R}{2}).$
\een
\end{theorem}

 By the relation between extrinsic and
intrinsic distances described in Theorem~\ref{thm1.1}, the extrinsic radius and curvature
estimates given in Theorems~\ref{rest} and~\ref{ext:cest},
that were used to prove Theorem~\ref{thm1.1}, become
radius and curvature estimates
that depend only on intrinsic distances; see~\cite{mt7}
for details on the proofs of the two theorems below.

In the next theorem the {\em radius} of a compact Riemannian surface $\Sigma$ with
boundary is the maximum intrinsic distance of points in the surface to its boundary;
in the second theorem, $d_{\S}$ denotes the intrinsic distance function of $\S$.

\begin{theorem}[Radius estimates, Theorem 1.2
in~\cite{mt7}] \label{rest3} There exists an
${\mathcal R}\geq \pi$ such that any compact disk
embedded in $\rth$ of constant mean curvature
$H>0$ has radius less than ${\mathcal R}/H$.
\end{theorem}

\begin{theorem}[Curvature Estimates, Theorem 1.3
in~\cite{mt7}] \label{cest3} Given $\delta$, $\cH>0$,
there exists a  $K(\delta,\cH)\geq\sqrt{2}\cH$ such
that  any compact  disk $M$ embedded in $\rth$ with
constant mean curvature $H\geq \cH$ satisfies
\[
\sup_{\large \{p\in {M} \, \mid \,
d_{M}(p,\partial M)\geq \delta\}} |A_{M}|\leq  K(\delta,\cH).
\]
\end{theorem}

An immediate consequence of the triangle inequality and
Theorem~\ref{thm1.1} is the following intrinsic version of
the one-sided curvature estimate given in Theorem~\ref{th}.
In the case that $H=0$, the next theorem follows from Corollary~0.8 in~\cite{cm35}.

\begin{theorem}[Intrinsic one-sided curvature estimate for $H$-disks] \label{TH}
There exist $\ve_I\in(0,\frac{1}{2})$
and $C_I \geq 2 \sqrt{2}$ such that for any $R>0$, the following holds.
Let $\cD$ be an $H$-disk such that $$\cD\cap \B(R)\cap\{x_3=0\}=\O $$
and $x\in \cD \cap \B(\ve_I R)$ where $d_\cD(x,\partial \cD) \geq R$.  Then:
\begin{equation} \label{EQ1}
|A_{\cD}|(x)\leq \frac{C_I}{R}.
\end{equation} In particular,  $H\leq \frac{C_I}{R}$.
\end{theorem}

\vspace{.3cm}
\center{William H. Meeks, III at profmeeks@gmail.com\\
Mathematics Department, University of Massachusetts, Amherst, MA 01003}
\center{Giuseppe Tinaglia at giuseppe.tinaglia@kcl.ac.uk\\ Department of
Mathematics, King's College London,
London, WC2R 2LS, U.K.}

\bibliographystyle{plain}
\bibliography{bill}

\end{document}